\documentclass{article}
\usepackage{graphicx} 
\usepackage{amsmath,amssymb,amsthm, epsfig}
\usepackage{color}
\usepackage{stackrel}
\usepackage{hyperref}
\usepackage{mathtools}
\usepackage{mathrsfs}
\usepackage{booktabs}
\usepackage[utf8]{inputenc}
\usepackage{comment}
\usepackage{fullpage}

\title{\Large\bf Bifurcation phenomena in one-phase quasilinear equations with nonstandard growth}
\date{\today}

\author{\it by \smallskip \\ Junior da Silva Bessa \footnote{\noindent Universidade Estadual de Campinas - UNICAMP. Instituto de Matem\'{a}tica, Estat\'{i}stica e Computa\c{c}\~{a}o Cient\'{i}fica - IMECC. Departamento  de Matemática. Bar\~{a}o Geraldo, Campinas - SP, Brazil. \noindent \texttt{E-mail address: \url{jbessa@unicamp.br}}} \,\, and\,\, Alan Pio Sousa 
 \footnote{\noindent  Universidade Federal do Cear\'{a}, Departamento de Matem\'{a}tica, Fortaleza - CE, Brazil \noindent \texttt{E-mail address: \url{alanpio@ufc.br }}}}

\newcommand{\intav}[1]{\mathchoice {\mathop{\vrule width 6pt height 3 pt depth  -2.5pt
\kern -8pt \intop}\nolimits_{\kern -6pt#1}} {\mathop{\vrule width
5pt height 3  pt depth -2.6pt \kern -6pt \intop}\nolimits_{#1}}
{\mathop{\vrule width 5pt height 3 pt depth -2.6pt \kern -6pt
\intop}\nolimits_{#1}} {\mathop{\vrule width 5pt height 3 pt depth
-2.6pt \kern -6pt \intop}\nolimits_{#1}}}
\newcommand{\defeq}{\mathrel{\mathop:}=}

\newtheorem{theorem}{Theorem}[section]

\newtheorem{lemma}[theorem]{Lemma}

\newtheorem{proposition}[theorem]{Proposition}

\newtheorem{corollary}[theorem]{Corollary}

\theoremstyle{definition}

\newtheorem{definition}[theorem]{Definition}

\theoremstyle{remark}

\newtheorem{remark}[theorem]{Remark}

\numberwithin{equation}{section}

\begin{document}
\maketitle
\begin{abstract}
\noindent The bifurcation phenomenon in a singularly perturbed one-phase free boundary problem governed by the g-Laplacian is established under prescribed boundary conditions. This bifurcation is characterized by the existence of a third solution, obtained via the Mountain Pass Lemma, when the boundary data decreases below a suitable threshold. Moreover, we analyze the asymptotic behavior of the associated evolution problem, proving convergence to stable stationary solutions and showing that the Mountain Pass solution is unstable in this dynamical sense.

\vspace{0.2cm}
	
\end{abstract}
\section{Introduction}

In this manuscript, we are interested in studying the following functional
\begin{equation}\label{eq1}
\mathcal{J}_{\varepsilon}[u]=\int_{\Omega}G(|\nabla u(x)|)+a(x)\Gamma_{\varepsilon}(u(x)) dx,
\end{equation}
which is a regularization of the Alt-Caffarelli type functional (cf. \cite{AltCaff81}) defined by
\begin{equation}\label{eq2}
\mathcal{J}[u]=\int_{\Omega}G(|\nabla u(x)|)+a(x)\chi_{\{u>0\}} dx.
\end{equation}
Here $\Omega\subset\mathbb{R}^{n}$ is a smooth bounded domain, $G:[0,\infty)\to [0,\infty)$ is a N-function with suitable properties (see the conditions {\bf (L1)-(L2)}), $\Gamma_{\varepsilon}\in C^{\infty}_{0}(\mathbb{R};[0,1])$ such that
\begin{equation*}
\Gamma_{\varepsilon}(s)=\left\{
\begin{array}{ll}
0 \quad \text{if} & s\leq 0 \\
1  \quad \text{if}& s\geq \varepsilon.
\end{array}
\right.
\end{equation*}
The weight function $a$ is assumed to be continuous, positive, and bounded in $\Omega$, and it satisfies $\displaystyle\inf_{x\in\Omega} a(x) > 0$. Observe that the Euler--Lagrange equation associated with \eqref{eq1} is given by
\begin{equation}\label{eq3}
-\Delta_{g}u+a(x)\beta_{\varepsilon}(u)=0\,\,\,\text{in}\,\,\, \Omega,
\end{equation}
where $\beta_{\varepsilon}(s)=\Gamma'_{\varepsilon}(s)$ and $\Delta_{g}$ is a \textit{$g$-Laplacian operator} defined by
$$
\Delta_{g}u:=\operatorname{div}\left(\frac{g(|\nabla u|)}{|\nabla u|}\nabla u\right).
$$
Moreover, we impose the boundary condition for a function $u$,
\begin{equation}\label{eq4}
u_{|_{\partial\Omega}}=\psi,
\end{equation}
for some function $\psi\in C^{1,\alpha}(\overline{\Omega})\cap W^{1,G}(\Omega)$ (see the definition of this functional spaces below). 

The interest in studying problem \eqref{eq1}, also known as a \textit{singular perturbation problem}, stems from the regularization procedure of the functional \eqref{eq2}, which yields, for each $\varepsilon>0$, a solution $u_{\varepsilon}\geq 0$ to \eqref{eq3}. As proved in \cite{MarWol09} (see also \cite{BraCarMor23,BraSou25}), the family of solutions $(u_{\varepsilon})_{\varepsilon>0}$ is precompact, and thus, up to a subsequence, it converges to a function $u$ that is a weak solution of the free boundary problem 
\begin{equation*}
\left\{
\begin{aligned}
\Delta_{g}u&= 0 && \text{in }  \{u>0\}\cap \Omega,\\
|\nabla u| &= \lambda^{*} && \text{on } \partial\{u>0\}\cap \Omega,
\end{aligned}
\right.
\end{equation*}
for some constant $\lambda^{*}$ depending on $g$ and $\mathrm{M}=\int_{0}^{1}\beta_{1}(s)ds$. In this case, the model arises naturally in combustion theory, specifically in flame propagation phenomena (cf. \cite{BerLar91,BerLArRoq89}).

Taking into account the above class of singular perturbation models and their connections with combustion theory, an interesting question concerns the existence and multiplicity of solutions associated with the corresponding nonlinear elliptic problems. We begin by recalling the seminal work of Caffarelli and Wang \cite{CaffWang15}, who proved a bifurcation phenomenon for the quadratic Alt--Caffarelli functional, obtaining a third mountain-pass type solution and analyzing the stability of the associated evolution problem.

Later, Ali and Wang \cite{HajWang19} (see also  \cite{LunCar23}) extended these results to the $p$-growth case $G(t)=\frac{t^{p}}{p}$, $1<p<\infty$, obtaining an analogous bifurcation phenomenon as well as convergence results for the associated evolution problem. In the singular regime $1<p<2$, they additionally assumed the dimensional restriction $n<\frac{p}{2-p}$.

Along these lines, we also mention the work of Charro et al. \cite{Charr}, who investigated a bifurcation phenomenon for a two-phase singularly perturbed free boundary problem governed by the Laplacian, arising in phase transition theory. They proved that uniqueness breaks down when the boundary data falls below a critical threshold value, and moreover established the existence of at least three distinct solutions in this regime, including a third mountain-pass type solution.

In view of the scarce literature on this problem, our work aims to establish a bifurcation phenomenon for \eqref{eq3}--\eqref{eq4}, governed by the magnitude of the boundary datum $\psi$. More precisely, we prove that if $\psi$ lies below a suitable threshold, then the problem has three distinct solutions: the trivial $g$-harmonic configuration, a global minimizer, and a saddle-type solution obtained via a Mountain Pass argument.

Throughout the paper we assume that the boundary datum satisfies
\[
\min_{\partial\Omega}\psi>\varepsilon>0,
\]
which guarantees that the trivial solution remains strictly above the activation region of the penalization term.

Furthermore, in the final part of this manuscript, we investigate the associated evolution problem
\begin{equation}\label{problemadeevoulucao}
\left\{
\begin{aligned}
w_{t}-\Delta_{g}w+a(x)\beta_{\varepsilon}(w)&= 0 && \text{in }  \Omega_{\infty}\coloneqq \Omega\times (0,+\infty),\\
    w(x,t) &= \psi && \text{on } S_{\infty}\coloneqq \partial \Omega\times (0,+\infty),\\
w(x,0)&= v_{0}(x)&& \text{for } x\in \overline{\Omega},
\end{aligned}
\right.
\end{equation}
where $v_{0}$ is a continuous function on $\overline{\Omega}$. In this setting, we prove the stability of both the minimizer and the trivial $g$-harmonic configuration, whereas the Mountain Pass solution turns out to be unstable.

Our governing operator, the $g$-Laplacian, can display either degenerate elliptic or singular behavior and, unlike the $p$-Laplacian, it is not necessarily homogeneous. This loss of homogeneity introduces extra challenges compared to the $p$-Laplacian setting and necessitates new arguments at several stages of the analysis. More specifically, we postulate that the function $g$ has different asymptotic profiles near $0$ and at infinity, a phenomenon absent in the $p$-Laplacian framework. To treat this general situation, we employ throughout the paper the structural hypotheses proposed by Lieberman \cite{Lieber}, which offer a natural setting for quasilinear problems with nonstandard growth (see Section \ref{Section2}).

The manuscript is organized as follows. In Section~\ref{Section2} we collect preliminary results concerning $N$-functions, Orlicz--Sobolev spaces, and the $g$-Laplacian operator. In Section~\ref{Section3} we establish the bifurcation phenomenon and obtain a third solution via the Mountain Pass Lemma. In Section~\ref{Section4} we prove a parabolic comparison principle, and in Section~\ref{Section5} we investigate the asymptotic behavior of the associated evolution problem.

\section{Preliminaries}\label{Section2}

We present in this section a collection of background results related to the N-functions and the  $g$-Laplacian operator.

First, we are adopt the conditions introduced by Lieberman in \cite{Lieber} for the study of elliptic models that may exhibit degenerate and/or singular behavior. Specifically, $G$ is a N-function satisfying the so-called \textit{Lieberman conditions} (cf. \cite{Lieber}):
\begin{itemize}
\item[{\rm (\textbf{L1})}]{\bf (Primitive condition)} The N-function $G$ satisfies 
$$
G'(t)=g(t),\quad \text{for} \quad g\in C^{0}([0,\infty))\cap C^{1}((0,\infty)).
$$ 
\item[{\rm(\textbf{L2})}] {\bf (Quotient condition)} There exist constants $0<\delta_0\le g_0$ such that
\[
\delta_0 \le \frac{t g'(t)}{g(t)} \le g_0 \qquad\forall\,t>0.
\]
\end{itemize}

These conditions provide the natural framework to deal with this class of nonstandard growth problems. As an illustrative example, if $G(t)=t^{p}$ for $1<p<\infty$, then $\delta_{0}=g_{0}=p-1$, respectively. Moreover, note that if $\delta_{0}=g_{0}$ in the condition {\rm (\textbf{L2})} then $g$ is a power function.

As illustrative examples, the following functions belong to the class considered in this article: 
\begin{itemize}
\item[(a)] \textbf{Double-phase type:}
\[
G(t)=\frac{t^{p}}{p}+\frac{t^{q}}{q}, \qquad g(t)=t^{p-1}+t^{q-1},
\qquad 1<p\le q<\infty.
\]
In this case, $\delta_{0}=p-1$ and $g_{0}=q-1$
\item[(b)] \textbf{Log-perturbed $p$-growth:}
\[
g(t)=t^{p-1}\bigl(1+\log(1+t)\bigr),\ 1<p<\infty.
\]
In this case, $\delta_{0}=p$ and $g_{0}=p+1$.
\end{itemize}
Furthermore, any positive linear combination of functions satisfying \textbf{(L2)} also satisfies \textbf{(L2)}. More precisely, if $g_{1}$ and $g_{2}$ satisfy \textbf{(L2)} with constants $(\delta_{i},g_{0,i})$, $i=1,2$, then for any $\alpha,\beta>0$ the function $g(t)=\alpha g_{1}(t)+\beta g_{2}(t)$ satisfies \textbf{(L2)} as well. Moreover, the product $g=g_{1}g_{2}$ satisfies \textbf{(L2)} with $\delta=\delta_{1}+\delta_{2}$ and $g_{0}=g_{0,1}+g_{0,2}$, while the composition $g(t)=g_{1}(g_{2}(t))$ satisfies \textbf{(L2)} with $\delta=\delta_{1}\delta_{2}$ and $g_{0}=g_{0,1}g_{0,2}$, thus highlighting the wide range of models encompassed by \eqref{eq2}, not only those exhibiting $p$-growth as in the classical examples mentioned above.

Under these assumptions, we have some basic properties that will be used throughout the paper, which are summarized in the following result.
\begin{proposition}
Let $G$ be a N-function that satisfies {\bf (L1)-(L2)}. Then, the function $g$ satisfies the following properties:

\begin{itemize}
\item[(g1)] $\min\{s^\delta, s^{g_0}\}g(t) \leq g(st) \leq \max\{s^\delta, s^{g_0}\}g(t),$
\item[(g2)] $G$ is convex and $C^2,$
\item[(g3)] $\dfrac{t g(t)}{1 + g_0} \leq G(t) \leq t g(t) \quad \forall t > 0.$
\item[(G1)] $\min\{s^{\delta+1}, s^{g_0+1}\} \frac{G(t)}{1 + g_0} \leq G(st) \leq (1 + g_0)\max\{s^{\delta+1}, s^{g_0+1}\}G(t)$.
\item[(G2)] $G(a + b) \leq 2^{g_0}(1 + g_0)\big(G(a) + G(b)\big), \quad \forall a,b > 0.$
\end{itemize}
\end{proposition}
\begin{proof}
For the proofs of (g1)--(g3), see \cite{Lieber}. Conditions (g1) and (g3) yield a corresponding inequality for $G$, which implies (G1). Furthermore, by the convexity of $G$ and this inequality, we deduce (G2).
\end{proof}

Associated with the $N$-function $G$, we consider the Orlicz and Orlicz--Sobolev spaces.

\begin{definition}
Given an $N$-function $G$, we define the Orlicz space $L^{G}(\Omega)$ as the set of measurable functions $h$ in $\Omega$ such that
\[
\rho_{G}(h)=\int_{\Omega}G(|h(x)|)\,dx<\infty,
\]
where we adopt the Luxemburg norm
\[
\|h\|_{L^{G}(\Omega)}\defeq\inf\left\{\lambda>0:\rho_{G}\left(\frac{h}{\lambda}\right)\leq 1\right\}.
\]
Associated with the Orlicz space, we define the Orlicz--Sobolev space $W^{1,G}(\Omega)$ as the set of all measurable functions $h$ such that $h$ and all its distributional derivatives $D_{i}h$, $i=1,\ldots,n$, belong to $L^{G}(\Omega)$. In this case, we endow $W^{1,G}(\Omega)$ with the norm
\[
\|h\|_{W^{1,G}(\Omega)}=\|h\|_{L^{G}(\Omega)}+\|\nabla h\|_{L^{G}(\Omega)}.
\]
\end{definition}

\begin{remark}
We observe that, under assumptions {\bf (L1)--(L2)}, $L^{G}(\Omega)$ and $W^{1,G}(\Omega)$ are reflexive Banach spaces (see for instance \cite{HarHas19}).
\end{remark}

The first property of the Orlicz spaces $L^{G}(\Omega)$ that we highlight here is the control of the norm by the modular $\rho_{G}$.

\begin{lemma}\label{behaviornorm}
If \(u\in L^{G}(\Omega)\), then there exists a positive constant \(\mathrm{C}=\mathrm{C}(\delta_{0},g_{0})\) such that 
\[
||u||_{L^G(\Omega)}\leq \mathrm{C}\max\left\{\left(\int_{\Omega}{G}(|u|)dx\right)^{\frac{1}{1+\delta_{0}}},\left(\int_{\Omega}{G}(|u|)dx\right)^{\frac{1}{1+g_{0}}}\right\}.
\]
\end{lemma}
\begin{proof}
See \cite[Lemma 2.3]{MarWol08}
\end{proof}

Moreover, the Orlicz--Sobolev space and the trace operator satisfy the following properties.

\begin{theorem}\label{embeddings}
Under the previous conditions, we have that:
\begin{itemize}
\item[(i)] $L^{G}(\Omega)\hookrightarrow L^{1+\delta_{0}}(\Omega)$ continuously;
\item[(ii)] There exists a constant $r\geq1$ depending only on $n$ and $\delta_0$ such that the trace operator $T:W^{1,1+\delta_0}(\Omega)\to L^{r}(\partial\Omega )$ is compact.
\end{itemize}
\end{theorem}

\begin{proof}
See Lemma 2.1 and Theorem 2.2 in \cite{MarWol08} for (i), and Theorem 6.2 in \cite{Necas} for (ii).
\end{proof}

Finally, in the next sections, we will need a Poincar\'e-type inequality in the framework of Orlicz--Sobolev spaces, which is summarized in the following result.
\begin{lemma}[\bf Poincar\'e-type inequality]\label{poincareinequality}
Let \(u\in W^{1,G}(\Omega)\) such that \(u=0\) on \(\partial \Omega\). Then, there exists a positive constant $\mathrm{C}$ depending only on $\operatorname{diam}(\Omega)$ such that 
\[
\int_{\Omega}{G}(|u|)\leq \int_{\Omega}G(\mathrm{C}|\nabla u|).
\]
\end{lemma}
\begin{proof}
See \cite[Lemma 2.4]{MarWol08}.
\end{proof}


\section{Multiplicity of the solutions: The bifurcation phenomena}\label{Section3}

In this section, we establish the existence of multiple solutions to the boundary value problem
\begin{equation}\label{eq5}
\left\{
\begin{aligned}
-\Delta_{g}u+a(x)\beta_{\varepsilon}(u)&= 0 && \text{in } \ \  \Omega \\
  u &= \psi && \text{on } \ \partial \Omega.
\end{aligned}
\right.
\end{equation}
It is worth emphasizing that any weak solution to the above problem is bounded by a constant depending only on $\Omega$, $g_0$, $\delta_0$, $g(1)$, $\|a\|_{\infty}$, and $\|\beta_\epsilon\|_{\infty}$, as established in \cite{Zheng}.
Initially, we consider the \textit{trivial solution} \(u_{0}\), obtained from the corresponding Dirichlet problem
\begin{equation}\label{max}
\left\{
\begin{aligned}
  -\Delta_{g} u &= 0 && \text{in } \Omega, \\
  u &= \psi && \text{on } \partial \Omega.
\end{aligned}
\right.
\end{equation}
By Theorem 1.1 in \cite{BraSou25}, we know that $u_0$ has $C^{1,\alpha}$ regularity up to the boundary.
Applying the Comparison Principle for  $g$-harmonic functions (cf. \cite[Theorem 2.2]{Braga1}) to the functions $u_0$, which is a solution of problem~\eqref{max}, and $v \equiv \varepsilon$, we obtain that $u_0 \geq \varepsilon$ in $\Omega$.
 Hence, for all $\phi\in C^{\infty}_{0}(\Omega)$ we have that
\begin{equation}
\int_{\Omega}\frac{g(|\nabla u_{0}|)}{|\nabla u_{0}|}\nabla u_{0}\cdot \nabla \phi +a(x)\beta_{\varepsilon}(u_{0})\varphi dx=0,
\end{equation}
because $\beta_{\varepsilon}(u_{0})=0$. Thus, $u_{0}$ is a solution to boundary value problem \eqref{eq5}.

Now we show the existence of a solution to the problem \eqref{eq5} via minimization of the functional \eqref{eq1}.
\begin{theorem}[\bf Minimizing Solution]\label{minfun}
There exists a minimizer for the functional $\mathcal{J}_{\varepsilon}$ on the admissible class
$$
\mathcal{K}_{\varepsilon}=\{v\in W^{1,G}(\Omega):\ v-\psi\in W^{1,G}_{0}(\Omega)\}.
$$
\end{theorem}
\begin{proof}
The proof of this result is standard. However, we include it here for the sake of completeness. First, note that $\mathcal{K}_{\varepsilon}\neq \emptyset$ because $\psi\in \mathcal{K}_{\varepsilon}$. Now, we can take a minimizing sequence $(u_{k})_{k\in\mathbb{N}}$ in $\mathcal{K}_{\varepsilon}$, i. e., the sequence satisfy
$$
\mathcal{J}_{\varepsilon}[u_{k}]\to \iota_{0}=\inf_{u\in\mathcal{K}_{\varepsilon}}\mathcal{J}_{\varepsilon}[u]\,\,\, \text{as}\,\,\, k\to \infty.
$$
So, the sequences $\int_{\Omega}G(|\nabla u_{k}|)$ and $\int_{\Omega}a(x)\Gamma_{\varepsilon}(u_{k})$ are bounded. Consequently,by the Poincaré-type inequality \ref{poincareinequality},  the sequence $(u_{k}-\psi)_{k\geq 1}$ is bounded in $W^{1,G}(\Omega)$. Thus, there exists a function $u_{\infty}\in W^{1,G}(\Omega)$ such that, up to a subsequence, $u_{k}\rightharpoonup u_{\infty}$ in this space. Moreover, by the Theorem \ref{embeddings} and the continuous embedding $W^{1,1+\delta_{0}}(\Omega)\hookrightarrow L^{1+\delta_{0}}(\Omega)$ we have that, up to a subsequence, $u_{k}\to u_{\infty}$ almost everywhere in $\Omega$ and $u_{\infty}=\psi$ on $\partial \Omega$. Thererfore, $u_{\infty}\in \mathcal{K}_{\varepsilon}$. Now, we analyzed the functional $\mathcal{J}_{\varepsilon}$ term by term:
\begin{itemize}
\item[\checkmark] For the sequence $\int_{\Omega}G(|\nabla u_{k}|)dx$, as G is a N-function by the weak convergence $u_{k}\rightharpoonup u_{\infty}$ in $W^{1,G}(\Omega)$ its follows that
$$
\int_{\Omega}G(|\nabla u_{\infty}|)dx\leq \liminf_{k\to \infty}\int_{\Omega}G(|\nabla u_{k}|)dx.
$$
\item[\checkmark] Now, for the sequence $\int_{\Omega}a(x)\Gamma_{\varepsilon}(u_{k})dx$, note that $\Gamma_{\varepsilon}(u_{k})\to \Gamma_{\varepsilon}(u_{\infty})$ a.e. in $\Omega$ and by $\Gamma_{\varepsilon}$ and $a$ are bounded functions, we can apply the Lebesgue’s Dominated Convergence Theorem and conclude that
$$
\int_{\Omega}a(x)\Gamma_{\varepsilon}(u_{k})dx\to \int_{\Omega}a(x)\Gamma_{\varepsilon}(u_{\infty})dx.
$$
\end{itemize}
Thus, $u_{\infty}$ satisfies
$$
\mathcal{J}_{\varepsilon}[u_{\infty}]\leq \liminf_{k\to \infty}\mathcal{J}_{\varepsilon}[u_{k}],
$$
as desired.
\end{proof}
\begin{remark}
From now on, we denote by $u_{2}$ the minimizer of $\mathcal{J}_{\varepsilon}$ in $\mathcal{K}_{\varepsilon}$ given by Theorem \ref{minfun}.     
\end{remark}
Now we show that, if the boundary data is sufficiently small, then the minimizer $u_{2}$ is different from the trivial solution $u_{0}$, thus giving rise to the \textit{bifurcation phenomenon}. Specifically, we adopt the following notation
$$
\mathrm{m}_{\psi}=\min_{\partial \Omega}\psi \,\,\, \text{and}\,\,\, \mathrm{M}_{\psi}=\max_{\partial\Omega}\psi.
$$
Our goal is to show that, under suitable assumptions on the boundary data $\psi$, one has $u_{2}\neq u_{0}$. Consider $u\in W^{1,G}(\Omega)$ a weak solution of the problem 
\begin{equation}\label{prob}
\left\{
\begin{aligned}
  u &= 0 && \text{on } \partial \Omega_\delta\\
    u &= \psi && \text{on } \partial \Omega\\
 -\Delta_{g} u &= 0 && \text{in } \Omega\setminus \Omega_{\delta},
\end{aligned}
\right.
\end{equation}
where $\Omega_{\delta}=\{x\in \Omega:\operatorname{dist}(x,\partial \Omega)>\delta\}$ and $\delta>0$ is a small constant that does not depend on $\varepsilon$ and $\psi$.  Consequently, the integral $\int_{\Omega_{\delta}} a(x)\,dx$ is also independent of these parameters. 

For the Laplacian and the \(p\)-Laplacian operators, the existence of estimates controlling the energy in terms of the supremum norm of \(\psi\) shows that the required condition is simply \(\mathrm{M}_{\psi} \ll 1\). In our case, however, the absence of such estimates leads us to impose a slightly stronger smallness assumption, namely, $\|\psi\|_{C^{0,1}(\Omega)} \ll 1.$

We now establish a boundary regularity result that will be instrumental in achieving our objective and naturally leads to the smallness condition imposed on $||\psi||_{C^{0,1}}.$ In what follows, we denote by $\mathrm{C}$ a positive constant depending only on the universal parameters $n, g_0, \delta_0$ and $G(1).$

\begin{lemma}\label{lema 1} Let $u\in W^{1,G}(\mathrm{B}_{r}^+(x_0))$ be a weak solution of 
\begin{equation}\label{eq1lemma3.3}
\left\{
\begin{aligned}
\Delta_{g}u&= 0 && \text{in } \ \  \mathrm{B}_{r}^{+}(x_{0})\defeq \mathrm{B}_{r}(x_{0})\cap \{x_{n}>0\} \\
  u &= \psi && \text{on } \ \partial\mathrm{B}^{+}_r(x_0).
\end{aligned}
\right.
\end{equation}
Then there exists $\mathrm{C}=\mathrm{C}(n, g_0, \delta_0)>0$ such that
$$\int_{\mathrm{B}_{r/2}^+(x_0)}G(|\nabla u|)\,dx \leq \mathrm{C}\int_{\mathrm{B}_{r}^+(x_0)}G(|\nabla\psi| + 4M_{\psi}/r)\,dx. $$    
\end{lemma}

\begin{proof}
First, recall in the definition of $\mathrm{M}_{\psi}$ and consider the following cutoff function $\varphi: \mathrm{B}^{+}_{1}\to\mathbb{R}$ defined by
\[
\varphi(x):=
\begin{cases}
0, 
& \text{if } x\in \mathrm|x-x_{0}|\leq \dfrac{r}{2},\\[6pt]
2\,\mathrm{M}_{\psi}\!\left(\dfrac{|x-x_{0}|-\frac{r}{2}}{\frac{r}{2}}\right),
& \text{if } \dfrac{r}{2}<|x-x_{0}|<r,\\[10pt]
2\,\mathrm{M}_{\psi},
& \text{if } |x-x_{0}|\geq r.
\end{cases}
\]
By the definition of $\varphi$, it follows immediately that $|\nabla \varphi|\leq \frac{4\mathrm{M}_\psi}{r}$. Now, we define the auxiliary function
\[
\eta:=\max\{\psi-u-\varphi,0\}.
\]
Clearly, $\eta\in W^{1,G}_0\big(\mathrm{B}_{r}^+(x_0)\big)$, since $u,\psi,\varphi\in W^{1,G}(\mathrm{B}^{+}_{r}(x_{0}))$ and $\eta=0$ on $\partial \mathrm{B}_{r}^{+}(x_{0}).$  Hence,
$$
0= \int_{\mathrm{B}_{r}^+(x_0)}\frac{g(|\nabla u|)}{|\nabla u|}\nabla u\cdot\nabla\eta = \int_{\mathrm{B}_{r}^+(x_0)\cap \{\eta>0\}}\frac{g(|\nabla u|)}{|\nabla u|}\nabla u\cdot(\nabla\psi - \nabla u - \nabla\varphi). 
$$
Moreover, 
       \begin{align*}
           \int_{\mathrm{B}_{r}^+(x_0)\cap \{\eta>0\}} g(|\nabla u|)|\nabla u| &\leq \int_{\mathrm{B}_{r}^+(x_0)\cap \{\eta>0\}}\frac{g(|\nabla u|)}{|\nabla u|}\nabla u\cdot(\nabla\psi - \nabla\varphi)\\
           &\leq \int_{\mathrm{B}_{r}^+(x_0)\cap \{\eta>0\}} g(|\nabla u|)|\nabla\psi-\nabla\varphi|.
       \end{align*}
       Using Young's inequality with $\epsilon$
       $$\int_{\mathrm{B}_{r}^+(x_0)\cap \{\eta>0\}} G(|\nabla u|)\leq \epsilon\int_{\mathrm{B}_{r}^+(x_0)\cap \{\eta>0\}}\tilde{G}(g(|\nabla u|)) +  \mathrm{C}(\epsilon)\int_{\mathrm{B}_{r}^+(x_0)\cap \{\eta>0\}}G(|\nabla\psi -\nabla\varphi|), $$
      where we take $\epsilon=\frac{1}{2g_0},$ then we get
      $$\int_{\mathrm{B}_{r}^+(x_0)\cap \{\eta>0\}}G(|\nabla u|)\,dx \leq \mathrm{C}\int_{\mathrm{B}_{r}^+(x_0)\cap \{\eta>0\}}G(|\nabla\psi-\nabla\varphi|)\,dx. $$
     Since $\mathrm{B}_{r/2}^+(x_0)\cap\{u-\psi<0\}\subset \mathrm{B}_{r}^+(x_0)\cap \{\eta>0\}$, it follows that
     $$\int_{\mathrm{B}_{r/2}^+(x_0)\cap \{v-\psi<0\}}G(|\nabla u|)\,dx \leq \mathrm{C}\int_{\mathrm{B}_{r}^+(x_0)}G(|\nabla\psi|+|\nabla\varphi|)\,dx. $$
   Obviously
    $$\int_{\mathrm{B}_{r/2}^+(x_0)\cap \{v=\psi\}}G(|\nabla u|)\,dx \leq \mathrm{C}\int_{\mathrm{B}_{r}^+(x_0)}G(|\nabla\psi|+|\nabla\varphi|)\,dx, $$
    so that 
    $$\int_{\mathrm{B}_{r/2}^+(x_0)\cap \{v-\psi\leq 0\}}G(|\nabla u|)\,dx \leq \mathrm{C}\int_{\mathrm{B}_{r}^+(x_0)}G(|\nabla\psi|+|\nabla\varphi|)\,dx. $$
    Similarly, taking
 $\eta =\max\{u-\psi -\varphi,0\}$ and once $\mathrm{B}_{r/2}^+(x_0)\cap\{u-\psi>0\}\subset \mathrm{B}_{r}^+(x_0)\cap\{\eta>0\}$ we also get
    $$\int_{\mathrm{B}_{r/2}^+(x_0)\cap \{v-\psi> 0\}}G(|\nabla u|)\,dx \leq \mathrm{C}\int_{\mathrm{B}_{r}^+(x_0)}G(|\nabla\psi|+|\nabla\varphi|)\,dx, $$
    Therefore,
    $$\int_{\mathrm{B}_{r/2}^+(x_0)}G(|\nabla u|)\,dx \leq \mathrm{C}\int_{\mathrm{B}_{r}^+(x_0)}G(|\nabla\psi|+|\nabla\varphi|)\,dx. $$
\end{proof}

Thus, for a function $u$ satisfying the assumptions of the lemma, we obtain that there exists a constant 
 $\mathrm{C}=\mathrm{C}(n, G(1),g_0,\delta_0)>0$ such that
$$
\int_{\mathrm{B}_{r/2}^+(x_0)}G(|\nabla u|)\,dx \leq \mathrm{C}\int_{\mathrm{B}_{r}^+(x_0)}G(|\nabla\psi|)\,dx + \mathrm{C}\max\{\mathrm{M}_{\psi}^{\delta_0+1},\mathrm{M}_{\psi}^{g_0+1}\}r^{-(1+g_0)}|\mathrm{B}_{r}^+|.
$$
On the other hand, by the interior estimate (cf. \cite[Lemma 2.6]{MarWol08}) for g-harmonic functions
 $$
 \int_{\mathrm{B}_{s}(y)}G(|\nabla u|)\,dx \leq \mathrm{C}\cdot\int_{B_{\frac{3}{2}}s(y)}G\left(\frac{|u|}{s}\right)\,dx.
 $$
Therefore, since by the Maximum Principle, we have
$\|u\|_{L^\infty}\leq \mathrm{C}\cdot \mathrm{M}_{\psi},$ thus in the interior we have
$$\int_{\mathrm{B}_{s}(y)}G(|\nabla u|)\,dx \leq \mathrm{C}\max\{\mathrm{M}_{\psi}^{1+\delta_0}, \mathrm{M}_{\psi}^{1+g_0}\}s^{-(1+g_0)}|\mathrm{B}_{\frac{3}{2}s}|. $$

\begin{remark}
It is well known that, with minor adjustments, the above estimates remain valid for weak solutions of $\mathrm{div}(\mathcal{A}(\nabla u)) = 0$, where the operator $\mathcal{A}$ satisfies the growth conditions introduced by G. Lieberman in \cite{Lieber}
\begin{enumerate}
\item $\mathcal{A}(p)\cdot p\geq C_1g(|p|)|p|$
\item $|\mathcal{A}(p)|\leq C_2g(|p|).$
\end{enumerate}
\end{remark}
By combining the interior and boundary estimates with standard covering arguments and the boundary flattening technique for $C^{1,\alpha}$ domains (see, for instance, \cite{Antoni}), we are able to derive the following global estimate for solutions to problem \eqref{prob}, 
\begin{eqnarray*}
\int_{\Omega\setminus\Omega_\delta}G(|\nabla u|)\,dx &\leq& \mathrm{C}\cdot \int_{\Omega\setminus \Omega_\delta}G(|\nabla\psi|)\,dx 
+ \mathrm{C(\delta)}\cdot\max\{\mathrm{M}_{\psi}^{\delta_0+1},\mathrm{M}_{\psi}^{g_0+1}\},
\end{eqnarray*}
where $\mathrm{C}(\delta)$ denotes a positive constant depending additionally on $\delta$. Since we have
 \begin{align*}
\mathcal{J}_{\varepsilon}[u]&\leq \int_{\Omega\setminus\Omega_\delta}G(|\nabla u|)\,dx + \int_{\Omega\setminus\Omega_\delta}a(x)\,dx,\\
 \end{align*} 
and as
 \[
\mathcal{J}_{\varepsilon}[u_{0}]
 =\int_{\Omega}\bigl(G(|\nabla u_{0}|)+a(x)\bigr)\,dx
 \geq \int_{\Omega}a(x)\,dx,
 \]
hence, for all small $\epsilon>0$
$$\mathcal{J}_\varepsilon[u] - \mathcal{J}_\varepsilon[u_0] \leq \int_{\Omega\setminus\Omega_\delta}G(|\nabla u|)\,dx - \int_{\Omega_\delta}a(x)\,dx <0, $$
if  $\mathrm{M}_{\psi}+ \int_{\Omega\setminus\Omega_\delta}G(|\nabla\psi|)\,dx\leq c_0, $ for some small constant
 $c_0=c_0(\delta,a).$ Observe that this condition is satisfied provided
\[
\mathrm{M}_{\psi}+ \|\nabla \psi\|_{L^{\infty}}\, G(1)\, |\Omega| \leq c_{0},
\]
which, in particular, holds whenever $\|\psi\|_{C^{0,1}} \ll 1.$

Finally, using the fact that $u_{2}$ is a minimizer of $\mathcal{J}_{\varepsilon}$, we conclude that
$$
\mathcal{J}_{\varepsilon}[u_{2}]\leq \mathcal{J}_{\varepsilon}[u]<\mathcal{J}_{\varepsilon}[u_{0}],
$$
proving that $u_{2}\neq u_{0}.$

\begin{remark}
The boundary data $\psi$ plays a crucial role in determining the number of solutions. In addition to the bifurcation phenomenon described above, if $\mathrm{m}_{\psi}$ is sufficiently large it is possible to see that the minimizing solution $u_{2}$ coincides with the trivial solution $u_{0}$. This observation follows, up to minor adjustments, the corresponding argument for the $p$-Laplacian in \cite{HajWang19}, relying instead on the $C^{1,\alpha}$ estimates for $g$-harmonic functions obtained in \cite{BraSou25}.
\end{remark}

This reveals a bifurcation phenomenon: when the boundary data is sufficiently large, the problem admits only the trivial solution \(u_0\). However, as the boundary data decreases toward a critical threshold, multiple solutions emerge, namely \(u_0\) and \(u_2\).

In this bifurcation setting, we construct a third solution to problem \eqref{eq5}. This claim is summarized in this result.

\begin{theorem}\label{thirdsolution}
If $\varepsilon\ll \mathrm{m}_{\psi}$ and $\mathcal{J}_{\varepsilon}[u_{2}]<\mathcal{J}_{\varepsilon}[u_{0}]$, then there exists a third weak solution $u_{1}$ of the problem \eqref{eq5}. Moreover, $\mathcal{J}_{\varepsilon}[u_{1}]\geq \mathcal{J}_{\varepsilon}[u_{0}]+\kappa$, for some $\kappa>0$ which is independent of $\varepsilon$. 
\end{theorem}

To this end, let $E = W^{1,G}_{0}(\Omega)$ endowed with the norm
\[
\|v\|_{E} := \|\nabla v\|_{L^{G}(\Omega)},
\]
which is well defined thanks to a Poincar\'e-type inequality (see \cite[Lemma 2.4]{MarWol08}), and makes $E$ a Banach space. Given $v \in E$, we adopt the convention $u = v + u_{0}$, so that
\[
\|v\|_{E} = \|\nabla u - \nabla u_{0}\|_{L^{G}(\Omega)}.
\]
We now introduce the following functional associated with $\mathcal{J}_{\varepsilon}$,
\begin{equation*}
\mathcal{I}_{\varepsilon}[v]=\mathcal{J}_{\varepsilon}[u]-\mathcal{J}_{\varepsilon}[u_{0}]=\int_{\Omega}G(|\nabla u|)-G(|\nabla u_{0}|)- \int_{\{u<\varepsilon\}}a(x)(1-\Gamma_{\varepsilon}(u)).
\end{equation*}
In other words, the functional $\mathcal{I}_{\varepsilon}$ measures the variation of the energy $\mathcal{J}_{\varepsilon}$ between the trivial solution $u_{0}$ and the perturbed profile $u = v + u_{0}$, taking into account both the change in the gradient term and the penalization on the active set $\{u > \varepsilon\}$. Clearly, $\mathcal{I}_{\varepsilon}[0] = 0$ and $\mathcal{I}_{\varepsilon}[v_{2}] \le 0$ for $v_{2} = u_{2} - u_{0}$, since $u_{2}$ is a minimizer of $\mathcal{J}_{\varepsilon}$ by Theorem~\ref{minfun}.

Assuming that the bifurcation phenomenon occurs, that is, $\mathcal{I}_{\varepsilon}[v_{2}] < 0$, we will show, by a topological method—the Mountain Pass Lemma—that this functional admits a critical point, from which we construct a weak solution to the Dirichlet problem \eqref{eq5}.

It is not difficult to verify that the Fr\'echet derivative of $\mathcal{I}_{\varepsilon}$ at $v \in E$ is given by
\begin{equation*}
\mathcal{I}_{\varepsilon}'[v]\varphi=\int_{\Omega}\frac{g(|\nabla u|)}{|\nabla u|}\nabla u\cdot \nabla \varphi-a(x)\beta_{\varepsilon}(u)\varphi, \, \varphi\in E,
\end{equation*}
where $\beta_{\varepsilon} = \Gamma_{\varepsilon}'$, and in this case $\mathcal{I}_{\varepsilon}'[v]$ can be viewed as an element of the dual space $E^{*}$. In summary, 
\[
\mathcal{I}_{\varepsilon}'[v]=-\Delta_{g}(v+u_{0})+a(x)\beta_{\varepsilon}(v+u_{0})\in E^{*}.
\]

We now verify that $\mathcal{I}_{\varepsilon}$ has $C^1$ regularity. Let $v\in E$ and $(w_n)$ any sequence that converges to $v$ in the $W^{1,G}$ norm. We have for any $\varphi\in E$
\begin{eqnarray}
|\mathcal{I}_{\varepsilon}'[v]\varphi-\mathcal{I}_{\varepsilon}'[w_n]\varphi
|&\leq& \left|\int_{\Omega}\left(\frac{g(|\nabla v+\nabla u_{0}|)}{|\nabla v+\nabla u_{0}|}(\nabla v+\nabla u_{0})-\frac{g(|\nabla w_n+\nabla u_{0}|)}{|\nabla w_n+\nabla u_{0}|}(\nabla w_n+\nabla u_{0})\right)\cdot \nabla \varphi\right|\nonumber\\
&+&\left|\int_{\Omega}a(x)(\beta_{\varepsilon}(v+u_{0})-\beta_{\varepsilon}(w_n+u_{0}))\varphi\right|\nonumber\\
&\coloneqq&I_{1}+I_{2}.\label{estlipder}
\end{eqnarray}
Since $L^G \hookrightarrow L^{1+\delta_0},$ we have $\nabla w_{n}\to \nabla v$ in norm $L^{1+\delta_0}.$ Then for a subsequence $(w_{n_k})$ we have that $|\nabla w_{n_k}|$ is uniform dominated in $L^{1+\delta_0}$ and $\nabla w_{n_k}\to \nabla v$ a.e. Thus, using Dominated Convergence Theorem, we get
$$\int_{\Omega}\frac{g(|\nabla w_{n_k}+\nabla u_{0}|)}{|\nabla w_{n_k}+\nabla u_{0}|}(\nabla w_{n_k}+\nabla u_{0})\cdot\nabla\varphi\,dx \to \int_{\Omega}\frac{g(|\nabla v+\nabla u_{0}|)}{|\nabla v+\nabla u_{0}|}(\nabla v+\nabla u_{0})\cdot\nabla\varphi\,dx.$$

On the other hand, once again using the fact that $L^G\hookrightarrow L^{1+\delta_0},$ together with the regularity of $\beta_\varepsilon,$ it follows from the Dominated Convergence Theorem that, for a subsequence $(w_{m_k})$ we have:
$$\int_\Omega a(x)\beta_\varepsilon(w_{m_k}+u_0)\varphi\,dx \to  \int_\Omega a(x)\beta_\varepsilon(v+u_0)\varphi\,dx.$$

Therefore, there always exists a subsequence such that, for any $\varphi \in E$, we have:
$$|\mathcal{I}_{\varepsilon}'[v]\varphi-\mathcal{I}_{\varepsilon}'[w_{n_k}]\varphi
| \to 0,
$$
and by a classical proof by contradiction argument, we see that this holds for the entire sequence $(w_{n})_{n \geq 1}$. Therefore, it follows that $\mathcal{I}'_\epsilon$ is continuous.

Next we justify the Palais–Smale condition on the functional $\mathcal{I}_\varepsilon$.
Suppose $(v_{k})_{k\geq 1}\subset E$ is a Palais–Smale sequence in the sense that
$$|\mathcal{I}_\varepsilon[v_k]|\leq \mathrm{M}
\quad\text{and}\quad \mathcal{I}'_\varepsilon[v_k]\to 0\quad\text{in}\,\,E^* $$
for some $\mathrm{M}>0$. Let $u_k=v_k+u_0\in W^{1,G}(\Omega)$. 

We observe that $a(x)\beta_\varepsilon(v+u_0)\in W^{1,G}_0(\Omega),$ since $a\in W^{2,G}$ is continuous and bounded, $\beta_\varepsilon$ and $\beta_\varepsilon'$ are smooth supported in $[0,\varepsilon],$
and $v+u_0>\varepsilon$ near $\partial\Omega.$ Since $L^G$ is continuous embedded in the Lebesgue space $L^{1+\delta_0},$ arguing in the same way as in \cite{HajWang19} with $p = 1 + \delta_0,$ we see that the mapping $v\mapsto a(x)\beta_\varepsilon(v+u_0)$ is a compact map from $W^{1,1+\delta_0}_0(\Omega)$ to $L^{(1+\delta_0)'}(\Omega)\subset E^*$ such that for a subsequence, still denoted by $(v_k), $ it holds that 
  $$ a(x)\beta_\varepsilon(v_k+u_0) \to -f\quad \text{in}\,\,\,L^{(1+\delta_0)'}(\Omega).$$
Here, by applying the result of \cite{HajWang19}, we obtain an analogous restriction on the dimension of the underlying space. Specifically, when \(0 < \delta_0 < 1\), we assume that
\[
n < \frac{1 + \delta_0}{1 - \delta_0}.
\]
We recall that
$$
|\mathcal{I}'_\varepsilon[v_k]|= \sup_{\|\varphi\|_{E}\leq 1}\left\vert \int_\Omega \frac{g(|\nabla u_k|)}{|\nabla u_k|}\nabla u_k\cdot\nabla\varphi + a(x)\beta_\varepsilon(u_k)\varphi\right\vert \to 0.
$$
As a consequence,
$$
\sup_{\|\varphi\|_{E}\leq 1}\left\vert \int_\Omega \frac{g(|\nabla u_k|)}{|\nabla u_k|}\nabla u_k\cdot\nabla\varphi -f\varphi\right\vert \to 0.
$$
Since the sequence $(\mathcal{I}_\varepsilon[v_k])_{k \geq 1}$ is bounded, it follows from Lemma 2.3 of \cite{MarWol08} that a subsequence still denoted by $(v_{k})_{k\geq 1}$  by abusing the notation without confusion, converges weakly in $E,$ and in particular weakly in $W^{1,1+\delta_0}_0(\Omega)$. Thus, we have
$$
\int_\Omega fv_k -fv_m \to 0\quad\text{as}\,\,\,k,m\to\infty.
$$
Then by setting $\varphi=v_k-v_m=u_k-u_m$ we get 
\begin{eqnarray}
\left\vert \int_\Omega\left(\frac{g(|\nabla u_k|)}{|\nabla u_k|}\nabla u_k - \frac{g(|\nabla u_m|)}{|\nabla u_m|}\nabla u_m\right)\cdot\nabla(u_k-u_m) \right\vert\leq \Bigg\vert \int_\Omega \frac{g(|\nabla u_k|)}{|\nabla u_k|}\nabla u_k\cdot\nabla(u_k-u_m)- f(u_k-u_m)\Bigg\vert\nonumber\\ 
 +\left\vert \int_\Omega \frac{g(|\nabla u_m|)}{|\nabla u_m|}\nabla u_m\cdot\nabla(u_k-u_m) - f(u_k-u_m)\right\vert \to 0. \label{eq 1}
\end{eqnarray}
We set $F(t)=\frac{g(t)}{t}$ and
 $$
 I^{k,m}=  \int_\Omega\left(\frac{g(|\nabla u_k|)}{|\nabla u_k|}\nabla u_k - \frac{g(|\nabla u_m|)}{|\nabla u_m|}\nabla u_m\right)\cdot\nabla(u_k-u_m)\,dx.
 $$
 Thus, arguing in the same way as in the proof of Theorem 1.1 in \cite{BraSou25}, we see that
\begin{align*}
    |I^{k,m}| &\geq \int_\Omega\int_0^1 F(|\nabla u_k+ (1-t)(\nabla u_k- \nabla u_m)|)|\nabla u_k -\nabla u_m|^2\,dt\,dx\\
    &\geq C(g_0,n)\left(\int_{S_1}F(|\nabla u_k|)|\nabla u_k -\nabla u_m|^2 +  \int_{S_2}G(|\nabla u_k -\nabla u_m|) \right)
\end{align*}
where $S_1=\{|\nabla u_k -\nabla u_m|\leq 2|\nabla u_k|\}$ and $S_2=\{|\nabla u_k -\nabla u_m|>2|\nabla u_k|\}.$

Thus, on the set $S_2$, we have
$$
\int_{S_2}G(|\nabla u_k -\nabla u_m|) \leq C|I^{k,m}| \to 0\quad\text{as}\,\,\,k,m\to\infty.
$$
Meanwhile, on the set $S_1$, since $g$ is increasing, we have
$$
F(|\nabla u_k-\nabla u_m|)|\nabla u_k -\nabla u_m| \leq \mathrm{C}F(|\nabla u_k|)|\nabla u_k|.
$$
Using Hölder’s inequalities, we see that
\begin{align*}
\int_{S_1}G(|\nabla u_k- \nabla u_m|) &\leq \int_{S_1}(F(|\nabla u_k- \nabla u_m|)|\nabla u_k- \nabla u_m|)^{\frac{1}{2}}F(|\nabla u_k- \nabla u_m|)^{\frac{1}{2}}|\nabla u_k- \nabla u_m|^{\frac{1}{2}}\\
&\leq \mathrm{C}\int_{S_1}(F(|\nabla u_k|)|\nabla u_k|)^{\frac{1}{2}}F(|\nabla u_k- \nabla u_m|)^{\frac{1}{2}}|\nabla u_k- \nabla u_m|^{\frac{3}{2}}\\
&\leq \mathrm{C}\left(\int_{S_1}F(|\nabla u_k|)|\nabla u_k- \nabla u_m|^{2} \right)^{\frac{1}{2}}\left(\int_{S_1}F(|\nabla u_k- \nabla u_m|)|\nabla u_k- \nabla u_m||\nabla u_k| \right)^{\frac{1}{2}}\\
&\leq \mathrm{C}\left(\int_{S_1}F(|\nabla u_k|)|\nabla u_k- \nabla u_m|^{2} \right)^{\frac{1}{2}}\left(\int_{S_1}g(|\nabla u_k|)|\nabla u_k| \right)^{\frac{1}{2}}\\
&\leq \mathrm{C}|I^{k,m}|^{\frac{1}{2}}\left(\int_{S_1}G(|\nabla u_k|)\right)^{\frac{1}{2}}.
\end{align*}
Since for $k=1, 2, 3,...$
$$
\int_{\Omega}G(|\nabla u_k|) \leq \mathrm{M} + \int_{\Omega}G(|\nabla u_0|)+ \|a\|_{\infty}=:\mathrm{C}(u_0, \mathrm{M}, a),
$$
we get
\begin{align*}
\int_{\Omega}G(|\nabla u_k-\nabla u_m|) &= \int_{S_1}G(|\nabla u_k-\nabla u_m|) + \int_{S_2}G(|\nabla u_k-\nabla u_m|)\\
&\leq \mathrm{C}|I^{k,m}|^{\frac{1}{2}}\mathrm{C}(u_0,\mathrm{M},a)^{\frac{1}{2}} + \mathrm{C}|I^{k,m}|,
 \end{align*}
and then, as a consequence of \eqref{eq 1}, it follows that
 $$
 \int_{\Omega}G(|\nabla u_k-\nabla u_m|) \to 0\quad\text{as}\,\,\,k,m\to \infty. 
 $$
Hence, it follows from Lemma \ref{behaviornorm} that $(u_k)_{k \geq 1}$ is a Cauchy sequence in $W^{1,G}(\Omega)$; consequently, $(v_k)_{k \geq 1}$ is a Cauchy sequence in $W^{1,G}_0(\Omega)$, and therefore it converges. The Palais–Smale condition is verified for the functional $\mathcal{I}_\epsilon$ on the Banach space $W^{1,G}_0(\Omega).$

We are now able to demonstrate that there exists a closed mountain ridge surrounding the origin of the Banach space $E$ that separates $v_2$ from the origin, with the energy $\mathcal{I}_\epsilon$ serving as the elevation function, as stated in the following lemma.

\begin{lemma}
For all small $\varepsilon>0$ such that $\mathrm{C}\epsilon \leq \frac{1}{2}\psi_m$ for a large universal constant $\mathrm{C}$, there exist positive constants $\delta$ and $\theta$ independent of $\varepsilon$, such that for every $v\in E$ with $\|v\|_E=\delta$, the inequality $\mathcal{I}_\varepsilon[v]\geq \theta$ holds.
\end{lemma}
\begin{proof}
 First, observe that it is sufficient to show that 
\[
\mathcal{I}_\varepsilon[v] \geq \theta > 0
\]
for any \( v \in C_0^{\infty}(\Omega) \) with \( \|v\|_E = \delta \), for sufficiently small \( \delta \). Indeed, this follows from the continuity of \( \mathcal{I}_\varepsilon \) in \( E \) and the density of \( C_0^{\infty}(\Omega) \) in \( E \).

Next, note that since \( L^G \) is continuously embedded into the Lebesgue space \( L^{1+\delta_0} \), the same argument used in \cite{HajWang19} applies here to show that for $\varepsilon$ small the set
\[
\Gamma = \{ x \in \Omega : u(x) \leq \varepsilon \},
\]
where \( u = v + u_0 \), is empty provided that \( \delta \) is sufficiently small.

Thus, since \( \Gamma = \emptyset \) and \( u_0 \) is the trivial solution, it follows from Theorem 2.3 in \cite{MarWol08} that we obtain the estimate
\begin{equation}\label{limporbai}
\mathcal{I}_\varepsilon[v]=\int_\Omega G(|\nabla(v+u_0) |) -G(|\nabla u_0|) \geq  \mathrm{C}\cdot\left(\int_{A_2}G(|\nabla v|) + \int_{A_1}F(|\nabla u +\nabla u_0|)|\nabla v|^2 \right),
\end{equation}
where $F(t)=\frac{g(t)}{t},$ and $A_1=\{x\in\Omega;|\nabla v|\leq 2|\nabla u_0 +\nabla v|\},\,A_2=\{x\in\Omega;|\nabla v|>2|\nabla u_0 +\nabla v|\}.$ Moreover, by Lemma \ref{behaviornorm}, we have
\[
\delta = \|\nabla v\|_{L^{G}(\Omega)} \leq \mathrm{C}_1 \max \left\{ \left( \int_\Omega G(|\nabla v|)\, dx \right)^{\frac{1}{1+g_0}}, \left( \int_\Omega G(|\nabla v|)\, dx \right)^{\frac{1}{1+\delta_0}} \right\}
\]
Let \( S = \{ x \in \Omega : |\nabla v| > \lambda \delta \} \), where \( \lambda > 0 \) is a constant to be determined, depending on \( |\Omega| \), $\delta_0, g_0,\,G(1).$
 Then, 
$$
\delta \leq \mathrm{C}_1 \max \left\{ \left( G(\lambda\delta)|\Omega| + \int_S G(|\nabla v|)\, dx \right)^{\frac{1}{1+g_0}}, \left( G(\lambda\delta)|\Omega|+ \int_S G(|\nabla v|)\, dx \right)^{\frac{1}{1+\delta_0}} \right\}.
$$
Therefore, for constants \( \mathrm{C}_1, \mathrm{C}_2 \) depending on \( g_0 \) and \( \delta_0 \), we have
$$
\mathrm{C}_1\max\{\delta^{1+\delta_0},\delta^{1+g_0}\}\leq \mathrm{C}_2\max\{\delta^{1+\delta_0},\delta^{1+g_0}\}G(\lambda)|\Omega| + \int_S G(|\nabla v|).
$$
In this way, we obtain
\begin{align*}
\int_S G(|\nabla v|) &\geq (\mathrm{C}_1 - \mathrm{C}_2|\Omega|G(\lambda))\max\{\delta^{1+\delta_0},\delta^{1+g_0}\}\\
&\geq \frac{1}{2}\max\{\delta^{1+\delta_0},\delta^{1+g_0}\}
\end{align*}
provided that \( \lambda > 0 \) is sufficiently small.

Hereafter, in search of having a lower bound for the right-hand side of estimate \eqref{limporbai}, we are interested in obtaining a lower bound for
$$
\int_{A_2\cap S}G(|\nabla v|)\,dx + \int_{A_1\cap S}F(|\nabla u_0+ \nabla v|)|\nabla v|^2\,dx. 
$$
Let us begin by dealing with the second integral
\begin{align*}
&\int_{A_1\cap S}F(|\nabla u_0+ \nabla v|)|\nabla v|^2 \geq  \mathrm{C}\int_{A_1\cap S}\frac{g(|\nabla v|)}{|\nabla v + \nabla u_0|}\cdot|\nabla v|^2\geq \mathrm{C}\int_{A_1\cap S}G(|\nabla v|)\frac{|\nabla v|}{|\nabla v| + |\nabla u_0|}\\
& = \mathrm{C}\int_{A_1\cap S\cap\{|\nabla u_0|\leq |\nabla v|\}}G(|\nabla v|)\frac{|\nabla v|}{|\nabla v| + |\nabla u_0|} + \mathrm{C}\int_{B_{2}}G(|\nabla v|)\frac{|\nabla v|}{|\nabla v| + |\nabla u_0|}.
\end{align*}
By simplicity of the notation, set $B_{1}= A_1\cap S\cap\{|\nabla u_0|\leq |\nabla v|\}$ and $B_{2}= A_1\cap S\cap\{|\nabla u_0|> |\nabla v|\}$. In this case,
$$
\int_{\mathrm{B}_{1}}G(|\nabla v|)\frac{|\nabla v|}{|\nabla v| + |\nabla u_0|} \geq \frac{1}{2}\int_{B_{1}}G(|\nabla v|),
$$
while on the other hand,
$$
\int_{B_{2}}G(|\nabla v|)\frac{|\nabla v|}{|\nabla v| + |\nabla u_0|}\geq \int_{B_{2}}G(|\nabla v|)\frac{|\nabla v|}{2|\nabla u_0|}.
$$
Applying H\"{o}lder's inequality, and using that \( G \) is an increasing function together with the Lipschitz continuity of \( u_0 \) in \( \Omega \),
\begin{align*}
\int_{B_{2}}G(|\nabla v|)|\nabla v|^{\frac{1}{2}} &\leq \left(\int_{B_{2}}G(|\nabla v|)\frac{|\nabla v|}{|\nabla u_0|}\right)^{\frac{1}{2}}\left(\int_{B_{2}}G(|\nabla v|)|\nabla u_0| \right)^{\frac{1}{2}} \\
&\leq \left(\int_{B_{2}}G(|\nabla v|)\frac{|\nabla v|}{|\nabla u_0|}\right)^{\frac{1}{2}}\left(\int_{B_{2}}G(|\nabla u_0|)|\nabla u_0| \right)^{\frac{1}{2}},
\end{align*}
consequently, 
$$
\int_{B_{2}}G(|\nabla v|)\frac{|\nabla v|}{2|\nabla u_0|} \geq 4\left(\int_{B_{2}} G(|\nabla v|)\right)^2\cdot\frac{\lambda\delta}{\left(\displaystyle\int_\Omega G(|\nabla u_0|)|\nabla u_0|\right)^{\frac{1}{2}}}.
$$
In summary, combining all these estimates, we conclude that for universal constants \( \mathrm{C}_3\), \(\mathrm{C}_4\), and \( \mathrm{C}_5 \),
\begin{align*}
\mathcal{I}_\varepsilon[v] &\geq \mathrm{C}_3\int_{A_2\cap S}G(|\nabla v|) + \mathrm{C}_4\int_{\mathrm{B}_{1}}G(|\nabla v|) + \left(\int_{B_{2}} G(|\nabla v|)\right)^2\cdot\frac{\mathrm{C}_5\lambda\delta}{\left(\displaystyle\int_\Omega G(|\nabla u_0|)|\nabla u_0|\right)^{\frac{1}{2}}}. 
\end{align*}
As in the case when
\begin{equation}\label{eq 2}
\min\left\{\int_{A_2\cap S}G(|\nabla v|), \,\int_{\mathrm{B}_{1}}G(|\nabla v|) \right\}\leq 1
\end{equation}
we have
\begin{align*}
\mathcal{I}_\epsilon[v] &\geq \min\left\{\mathrm{C}_3, \mathrm{C}_4, \frac{\mathrm{C}_5\lambda\delta}{\left(\int_\Omega G(|\nabla u_0|)|\nabla u_0|\right)^{\frac{1}{2}}} \right\}\cdot\left(\int_SG(|\nabla v|) \right)^2\\
&\geq  \min\left\{\mathrm{C}_3, \mathrm{C}_4, \frac{\mathrm{C}_5\lambda\delta}{\left(\int_\Omega G(|\nabla u_0|)|\nabla u_0|\right)^{\frac{1}{2}}} \right\}\cdot\left(\frac{1}{2}\max\{\delta^{1+\delta_0},\delta^{1+g_0}\} \right)^2=:\mathrm{A}(\delta, u_0).
\end{align*}
If \eqref{eq 2} does not hold, then
$$
\mathcal{I}_\varepsilon[v]\geq \min\{\mathrm{C}_3, \mathrm{C}_4\}.
$$
Therefore, in either situation, we have
$$
\mathcal{I}_\varepsilon[v]\geq \min\{\mathrm{C}_3, \mathrm{C}_4, \mathrm{A}(\delta, u_0)\}=:\theta>0.
$$
This completes the proof.   
\end{proof}
We are now in a position to prove the bifurcation phenomenon. To this end, let
\[
\mathcal{G} = \{\gamma \in C([0,1], E) : \gamma(0) = 0 \text{ and } \gamma(1) = v_2\}
\]
and
\[
c = \inf_{\gamma \in \mathcal{G}} \max_{0 \leq t \leq 1} \mathcal{I}_\varepsilon[\gamma(t)].
\]
The verified Palais--Smale condition and the preceding lemma allow us to apply the classical Mountain Pass Theorem as stated, in \cite{Youss} to conclude that there is a \( v_1 \in {E} \) such that \( \mathcal{I}_\varepsilon[v_1] = c \), and \( \mathcal{I}_\varepsilon'[v_1] = 0 \) in \( {E}^* \). That is,
$$
\int_\Omega\frac{g(|\nabla u_1|)}{|\nabla u_1|}\nabla u_1\cdot\nabla\varphi - a(x)\beta_\varepsilon(u_1)\varphi\,dx=0,
$$
for any $\varphi\in E=W^{1,G}_0(\Omega),$ where $u_1=v_1+u_0.$ So $u_1$ is a weak solution of the problem \eqref{eq3} and \eqref{eq4}. In essence, the Mountain Pass Theorem provides a method for obtaining a saddle-type solution. Consequently, \( u_1 \) is typically unstable, in contrast to the stable solutions \( u_0 \) and \( u_2 \). This proves Theorem \ref{thirdsolution}.

\section{A parabolic comparison principle}\label{Section4}

In this part, we turn our attention to the following evolution problem
\begin{equation}\label{evprob}
\left\{
\begin{aligned}
\mathfrak{L}_{g}(w) &= 0 && \text{in }  \Omega_{\infty}\coloneqq \Omega\times (0,+\infty)\\
    w(x,t) &= \psi && \text{on } S_{\infty}\coloneqq \partial \Omega\times (0,+\infty)\\
w(x,0)&= v_{0}(x)&& \text{for } x\in \overline{\Omega},
\end{aligned}
\right.
\end{equation}
where 
$$
\mathfrak{L}_{g}(w)\coloneqq w_{t}-\Delta_{g}w+\zeta(x,w)
$$
is a parabolic operator with a lower-order term $\zeta$ which is continuous and such that
\begin{equation}\label{condzeta}
|\zeta(x,r_{1})-\zeta(x,r_{2})|\leq \mathrm{C}_{\zeta}|r_{1}-r_{2}|,\, \forall x\in \Omega,\ r_{1},r_{2}\in \mathbb{R},
\end{equation}
and $\mathrm{C}_{\zeta}\geq 0$. 

In order to analyze problem \eqref{problemadeevoulucao}, which is a particular case of \eqref{evprob}, we develop a comparison principle between sub- and super-solutions for problem \eqref{evprob}. Here, we adopt the notion of weak sub- and super-solutions as presented in the following definition.
\begin{definition}
We say that $w\in L^{2}_{loc}((0,+\infty);W^{1,G}(\Omega))$ is a weak sub-solution (respectively, super-solution) of \eqref{evprob}, if for any $T>0$ any every test function $0\leq \varphi\in L^{2}_{0}(0,T;W^{1,G}(\Omega))$ such that $\varphi_{t}\in L^{2}(\Omega\times (0,T])$,
\begin{equation*}
\int_{0}^{T}\int_{\Omega}\left(-w\varphi_{t}+\frac{g(|\nabla w|)}{|\nabla w|}\nabla w\cdot \nabla \varphi+\zeta(x,w)\varphi\right)\leq (\geq) \ 0,
\end{equation*}
where $L_0^2(0,T; W_0^{1,G}(\Omega))$ is the subset of $L^2(0,T,W^{1,G}(\Omega))$ that contains functions that are equal to zero on the boundary of $\Omega \times [0,T].$
Moreover, $w_{t}\in L^{2}_{loc}(\Omega\times (0,+\infty))$, $w(\cdot,t)=\psi$ on $\partial \Omega$ in the trace sense and $w(\cdot,t)\to v_{0}$ in $L^{2}$ when $t\to 0$. Finally, we say that \(w\) is a weak solution to \eqref{evprob} if it is both a sub-solution and a super-solution in the weak sense.
\end{definition}
\begin{remark}
 Unless otherwise stated, we will refer to sub-, super-, and solutions omitting the word ``weak''.   
\end{remark}
The main result of this section is summarized in the following theorem.
\begin{theorem}[\bf Parabolic Comparison Principle]\label{comppri}
Consider $w_{1}$ and $w_{2}$ sub- and super-solutions of \eqref{evprob}, respectively, and assume that $\zeta$ satisfies the condition \eqref{condzeta}. If $w_{1}\leq w_{2}$ on $\partial_{p}\Omega_{\infty}\coloneqq (\overline{\Omega}\times\{0\})\cup S_{\infty}$, then $w_{1}\leq w_{2}$ in $\Omega_{\infty}$.
\end{theorem}
As an immediate consequence of the Parabolic Comparison Principle \ref{comppri}, we obtain uniqueness of solutions for the evolution problem \eqref{evprob}.
\begin{corollary}[\bf Uniqueness]\label{Uniquenesscor}
Consider $w_{1}$ and $w_{2}$ solutions of \eqref{evprob}, respectively, and assume that $\zeta$ satisfies the condition \eqref{condzeta}. Then $w_{1}= w_{2}$.
\end{corollary}
To prove the desired result, we first require a comparison principle on cylinders with a small time interval.This is established in the following lemma.
\begin{lemma}\label{comppriest}
Assume that $\zeta$ satisfies condition \eqref{condzeta}. For $T>0$ sufficiently small, if
\[
\mathfrak{L}_{g}(w_{1})\leq 0\leq \mathfrak{L}_{g}(w_{2}) \quad \text{in } \Omega_{T}\coloneqq \Omega\times [0,T],
\]
and $w_{1}<w_{2}$ on $\partial_{p}\Omega_{T}$, then $w_{1}\leq w_{2}$ in $\Omega_{T}$.
\end{lemma}
\begin{proof}
First, observe that we may assume $\mathfrak{L}_{g}(w_{1})<0$. Indeed, otherwise, for small $\eta>0$ consider the auxiliary function
$$
\bar{w}_{1}(x,t)=w_{1}(x,t)-\frac{\eta}{T-t},\, x\in \Omega,\, t\in[0,T).
$$
Note that $\lim_{t\to T^-}\bar{w}_{1}(x,t)=-\infty$ uniformly in $\Omega$ and that $\bar{w}_{1}<w_{2}$ on the parabolic boundary $\partial_{p}\Omega_{T}$. Moreover, it is not difficult to verify that $\bar{w}_{1}$ satisfies
$$
\mathfrak{L}_{g}(\bar{w})\leq -\frac{\eta}{2(T-t)^{2}}\leq -\frac{\eta}{2T^{2}}<0,
$$
whenever, $T\leq \frac{1}{2\mathrm{C}_{\zeta}}$. Therefore, replacing $w_{1}$ by $\bar{w}_{1}$ if necessary, we may assume that $\mathfrak{L}_{g}(w_{1})\leq\frac{-\eta}{2T^{2}}<0$.

With this preliminary observation in hand, define $v_{j}(x,t)=e^{-\alpha t}w_{j}(x,t)$, $j=1,2$, for $\alpha>2\mathrm{C}_{\zeta}$. To prove the desired result, it is enough to show that $v_{1}\leq v_{2}$ in $\Omega_{T}$. Using that $w_{1}$ and $w_{2}$ satisfy
\[
\mathfrak{L}_{g}(w_{1})\leq -\frac{\eta}{2T^{2}}<0\leq \mathfrak{L}_{g}(w_{2}),
\]
we obtain that for every test function $\varphi$, the functions $v_{1}$ and $v_{2}$ satisfy
\[
\int_{0}^{T}\int_{\Omega}\left(-e^{\alpha t}v_{1}\varphi_{t}+\frac{g(e^{\alpha t}|\nabla v_{1}|)}{|\nabla v_{1}|}\nabla v_{1}\cdot \nabla \varphi +\zeta(x,e^{\alpha t}v_{1})\right)\leq -\int_{0}^{T}\int_{\Omega}\frac{\eta}{2T^{2}}\varphi
\]
and
\[
\int_{0}^{T}\int_{\Omega}\left(-e^{\alpha t}v_{2}\varphi_{t}+\frac{g(e^{\alpha t}|\nabla v_{2}|)}{|\nabla v_{2}|}\nabla v_{2}\cdot \nabla \varphi +\zeta(x,e^{\alpha t}v_{2})\right)\geq 0.
\]
Subtracting these two inequalities term by term, we obtain that
\begin{eqnarray*}
\int_{0}^{T}\int_{\Omega}\Bigg[-e^{\alpha t}(v_{1}-v_{2})\varphi_{t}+\left(\frac{g(e^{\alpha t}|\nabla v_{1}|)}{|\nabla v_{1}|}\nabla v_{1}-\frac{g(e^{\alpha t}|\nabla v_{2}|)}{|\nabla v_{2}|}\nabla v_{2}\right)\cdot \nabla \varphi\\
+(\zeta(x,e^{\alpha t}v_{1})-\zeta(x,e^{\alpha t}v_{2}))\varphi\Bigg]\leq -\int_{0}^{T}\int_{\Omega}\frac{\eta}{2T^{2}}\varphi.
\end{eqnarray*}
In particular, using $\varphi=(v_{1}-v_{2})^{+}=\max\{v_{1}-v_{2},0\}$, we obtain
\begin{eqnarray}\label{eqI}
I&=&\int_{0}^{T}\int_{\{v_{1}>v_{2}\}}\Bigg[-e^{\alpha t}(v_{1}-v_{2})(v_{1}-v_{2})_{t}+\left(\frac{g(e^{\alpha t}|\nabla v_{1}|)}{|\nabla v_{1}|}\nabla v_{1}-\frac{g(e^{\alpha t}|\nabla v_{2}|)}{|\nabla v_{2}|}\nabla v_{2}\right)\cdot \nabla(v_{1}-v_{2}) \nonumber\\
&+&(\zeta(x,e^{\alpha t}v_{1})-\zeta(x,e^{\alpha t}v_{2}))(v_{1}-v_{2})\Bigg]\leq -\frac{\eta}{2T^{2}}\int_{0}^{T}\int_{\{v_{1}>v_{2}\}}(v_{1}-v_{2}).
\end{eqnarray}
Next, we analyze the integral $I$. First, note that
\begin{eqnarray}
I_{1}&:=&\int_{0}^{T}\int_{\{v_{1}>v_{2}\}}\left(\frac{g(e^{\alpha t}|\nabla v_{1}|)}{|\nabla v_{1}|}\nabla v_{1}-\frac{g(e^{\alpha t}|\nabla v_{2}|)}{|\nabla v_{2}|}\nabla v_{2}\right)\cdot \nabla(v_{1}-v_{2})\nonumber\\
&=&\int_{0}^{T}\int_{\{v_{1}>v_{2}\}}\Bigg[g(e^{\alpha t}|\nabla v_{1}|)|\nabla v_{1}|+g(e^{\alpha t}|\nabla v_{2}|)|\nabla v_{2}|-(g(e^{\alpha t}|\nabla v_{1}|)|\nabla v_{2}|\nonumber\\
&+&g(e^{\alpha t}|\nabla v_{2}|)|\nabla v_{1}|)\frac{\nabla v_{1}}{|\nabla v_{1}|}\cdot \frac{\nabla v_{2}}{|\nabla v_{2}|}\Bigg]\nonumber\\
&\geq&\int_{0}^{T}\int_{\{v_{1}>v_{2}\}}\Big[g(e^{\alpha t}|\nabla v_{1}|)|\nabla v_{1}|+g(e^{\alpha t}|\nabla v_{2}|)|\nabla v_{2}|-(g(e^{\alpha t}|\nabla v_{1}|)|\nabla v_{2}|+g(e^{\alpha t}|\nabla v_{2}|)|\nabla v_{1}|)\Big]\nonumber\\
&=&\int_{0}^{T}\int_{\{v_{1}>v_{2}\}}(g(e^{\alpha t}|\nabla v_{1}|)-g(e^{\alpha t}|\nabla v_{2}|))(|\nabla v_{1}|-|\nabla v_{2}|)\nonumber\\
&\geq&0,\label{est1}
\end{eqnarray}
since $g$ is a non-decreasing function. In addition, for the integral 
$$
I_{2}:=\int_{0}^{T}\int_{\{v_{1}>v_{2}\}}-e^{\alpha t}(v_{1}-v_{2})(v_{1}-v_{2})_{t}
$$
we observe that $\{v_{1}>v_{2}\}\subset \Omega\times (0,T)$, since $v_{1}\leq v_{2}$ on $\partial_{p}\Omega_{T}$ (recall that $w_{1}\leq w_{2}$ on $\partial_{p}\Omega_{T}$) and $v_{1}\to -\infty$ as $t\to T^{-}$, and thus, we can apply the divergence theorem and conclude that
\begin{equation}\label{est2}
I_{2}=\int_{0}^{T}\int_{\{v_{1}>v_{2}\}}\frac{\alpha}{2} e^{\alpha t}(v_{1}-v_{2})^{2}. 
\end{equation}
Finally, by the Lipschitz condition \eqref{condzeta}, it follows that
\begin{eqnarray}
I_{3}&:=&\int_{0}^{T}\int_{\{v_{1}>v_{2}\}}(\zeta(x,e^{\alpha t}v_{1})-\zeta(x,e^{\alpha t}v_{2}))(v_{1}-v_{2})
\geq\int_{0}^{T}\int_{\{v_{1}>v_{2}\}}-\mathrm{C}_{\zeta}e^{\alpha t}(v_{1}-v_{2})^{2}\label{est3}.
\end{eqnarray}
By the choice of $\alpha$ and estimates \eqref{est1}, \eqref{est2}, and \eqref{est3}, we can conclude that
\begin{eqnarray*}
-\frac{\eta}{2T^{2}}\int_{0}^{T}\int_{\{v_{1}>v_{2}\}}(v_{1}-v_{2})\geq I=I_{1}+I_{2}+I_{3}\geq\int_{0}^{T}\int_{\{v_{1}>v_{2}\}}\left(\frac{\alpha}{2}-\mathrm{C}_{\zeta}\right)e^{\alpha t}(v_{1}-v_{2})^{2}\geq0
\end{eqnarray*}
and this estimate can only occur when the measure of the set $\{v_{1}>v_{2}\}$ is zero. This completes the proof of the desired result.
\end{proof}
The strict sign condition in the hypothesis of Lemma \ref{comppriest} is not restrictive, and in practice we can remove it, as shown in the following lemma.
\begin{lemma}\label{compprilocal}
Assume that $\zeta$ satisfies condition \eqref{condzeta}. For $T>0$ sufficiently small, if
\[
\mathfrak{L}_{g}(w_{1})\leq 0\leq \mathfrak{L}_{g}(w_{2}) \quad \text{in } \Omega_{T},
\]
and $w_{1}\leq w_{2}$ on $\partial_{p}\Omega_{T}$, then $w_{1}\leq w_{2}$ in $\Omega_{T}$.
\end{lemma}
\begin{proof}
Given $\eta>0$, choose $\tilde{\eta}>0$ such that $\tilde{\eta}\leq \dfrac{\eta}{4\mathrm{C}_{\zeta}}$, and define
\[
\hat{w}_{1}(x,t):=w_{1}(x,t)-\eta t+\tilde{\eta}, \qquad (x,t)\in \overline{\Omega_{T}}.
\]
Clearly, $\hat{w}_{1}<w_{1}\leq w_{2}$ on $\partial_{p}\Omega_{T}$. Moreover, for any test function $\varphi$, since $w_{1}$ is a subsolution of $\mathfrak{L}_{g}(w_{1})=0$, we have
\begin{align*}
\int_{0}^{T}\!\!\int_{\Omega}
\left(
 -\hat{w}_{1}\varphi_{t}
 +\frac{g(|\nabla \hat{w}_{1}|)}{|\nabla \hat{w}_{1}|}\nabla \hat{w}_{1}\cdot\nabla \varphi
 +\zeta(x,\hat{w}_{1})\varphi
\right)
&\leq
\int_{0}^{T}\!\!\int_{\Omega}
\left((\zeta(x,\hat{w}_{1})-\zeta(x,w_{1}))\varphi+(\eta t+\tilde{\eta})\varphi_{t}\right)\\
&\leq
\int_{0}^{T}\!\!\int_{\Omega}
\big[\mathrm{C}_{\zeta}(\eta t+\tilde{\eta})-\eta\big]\varphi\\
&\leq
\int_{0}^{T}\!\!\int_{\Omega}
\left[\mathrm{C}_{\zeta}\!\left(\eta T+\frac{\eta}{4\mathrm{C}_{\zeta}}\right)-\eta\right]\varphi\\
&\leq
\int_{0}^{T}\!\!\int_{\Omega}
\left(-\eta+\frac{\eta}{2}+\frac{\eta}{4}\right)\varphi
= -\frac{\eta}{4}\int_{0}^{T}\!\!\int_{\Omega}\varphi
\leq 0,
\end{align*}
whenever $T<\dfrac{1}{2\mathrm{C}_{\zeta}}$. Hence $\hat{w}_{1}$ is a strict subsolution in $\Omega_{T}$, with $\hat{w}_{1}\leq w_{2}$ on $\partial_{p}\Omega_{T}$. By Lemma \ref{comppriest}, it follows that $\hat{w}_{1}\leq w_{2}$ in $\overline{\Omega_{T}}$ for all sufficiently small $T>0$, for every small $\eta>0$ and $\tilde{\eta}$ as above. Letting $\eta\to 0^{+}$ (and hence $\tilde{\eta}\to 0^{+}$) yields $w_{1}\leq w_{2}$ in $\overline{\Omega_{T}}$.
\end{proof}
Now we can prove the comparison principle \ref{comppri}.
\begin{proof}[\bf Proof of Theorem \ref{comppri}]
Let $T_{0}$ be any value smaller than the constant $T$ in Lemma~\ref{compprilocal}. Then, by this lemma, $w_{1} \leq w_{2}$ in $\overline{\Omega_{T_{0}}}$, and in particular $w_{1} \leq w_{2}$ on $\partial_{p}(\Omega \times (T_{0},2T_{0}))$. Applying the preceding lemma once more, we conclude that $w_{1} \leq w_{2}$ in $\overline{\Omega \times (T_{0},2T_{0})}$, and so on. Repeating this argument inductively, we obtain the desired conclusion.
\end{proof}
\begin{remark}
We point out that an analogous result remains valid when the evolution problem \eqref{evprob} is posed on a finite time interval, say $t \in [0,T]$.
\end{remark}

\section{Convergence of evolution problem}\label{Section5}

In this final part we are interested in applying the Comparison Principle~\ref{comppri} to study the following evolution problem:

\begin{equation}\label{evprobprincipal}
\left\{
\begin{aligned}
w_{t}-\Delta_{g}w+a(x)\beta_{\varepsilon}(w)&= 0 && \text{in }  \Omega_{\infty}\coloneqq \Omega\times (0,+\infty)\\
    w(x,t) &= \psi && \text{on } S_{\infty}\coloneqq \partial \Omega\times (0,+\infty)\\
w(x,0)&= v_{0}(x)&& \text{for } x\in \overline{\Omega}.
\end{aligned}
\right.
\end{equation}
which is an example of the evolution problem \eqref{evprob} with
\(\zeta(x,s) = a(x)\beta_{\varepsilon}(s)\), and which satisfies the
assumptions imposed in Section~\ref{Section4}. As we metioned, as a consequence of the Parabolic Comparison Principle \ref{comppri}, we obtain uniqueness of solutions for the evolution problem.
We will prove a convergence theorem for the evolution associated with this model, which links it directly to problem~\eqref{eq5}.

To this end, we introduce the following terminology. We denote by
$\mathcal{WS}$ the class of weak solutions of the stationary problem
\eqref{eq5}. Within this class, the trivial solution $u_{0}$ plays the role of a maximal element, while the function $u_{2}$ is the minimal one and can be obtained as the infimum of all supersolutions. Observe that, by the Parabolic Comparison Principle, if the boundary data satisfy $u_2 \leq v_0 \leq u_0$, then the solution $w$ of problem \ref{evprobprincipal} satisfies $u_2(x) \leq w(x,t) \leq u_0(x)$ in $\Omega \times (0,+\infty)$.

Moreover, we say that a function $u$ is a non-minimal solution of problem \eqref{eq5} if it is a viscosity solution but fails to be a local minimizer of the functional $\mathcal{J}_{\varepsilon}$ in the following sense: for every $\eta>0$ there exists an admissible competitor $v$ for $\mathcal{J}_{\varepsilon}$, with $v=\psi$ on $\partial\Omega$ and $\|v-u\|_{L^{\infty}(\Omega)}<\eta$, such that $\mathcal{J}_{\varepsilon}[v]<\mathcal{J}_{\varepsilon}[u]$.

We are now in a position to state the main result of this section.
\begin{theorem}[\bf Convergence of evolution]
Let $w$ be a solution of \eqref{evprobprincipal}. If the initial data
$v_{0}$ belongs to any of the classes listed below, then the corresponding
assertion holds.
\begin{itemize}
\item[(i)] If $v_{0}\leq u_{2}$ on $\overline{\Omega}$, then
\begin{equation}\label{condlimu2}
\lim_{t\to +\infty}w(x,t)=u_{2}(x)\,\,\, \text{locally uniformly for } x\in\overline{\Omega};
\end{equation}
\item[(ii)] If $u_{2}< \bar{u}_{2}$, where
\[
\bar{u}_{2}(x)=\inf\limits_{\substack{
u \in \mathcal{WS} \\
u \geq u_{2} \\
u \neq u_{2}
}} u(x),\, x\in \overline{\Omega},
\]
then for $v_{0}$ such that $u_{2}<v_{0}<\bar{u}_{2}$ imply in the convergence \eqref{condlimu2};
\item[(iii)] If $u_{0}> \bar{u}_{0}$, where
\[
\bar{u}_{0}(x)=\sup\limits_{\substack{
u \in \mathcal{WS} \\
u \leq u_{0} \\
u \neq u_{0}
}} u(x),\, x\in \overline{\Omega},
\]
then for $v_{0}$ such that $\bar{u}_{0}<v<u_{0}$,
\begin{equation}\label{condlimu0}
\lim_{t\to +\infty}w(x,t)=u_{0}(x)\,\,\, \text{locally uniformly for } x\in\overline{\Omega};
\end{equation}
\item[(iv)] If $v_{0}\geq u_{0}$, then the convergence \eqref{condlimu0}  hold;
\item[(v)] Consider $u_{1}$ the non-minimal solution  to \eqref{eq5}. For any small $\eta>0$, there exists $v_{0}$ such that $\|v_{0}-u_{1}\|_{L^{\infty}(\Omega)}<\eta$, however the solution $w$ to evolutionary problem \eqref{evprobprincipal} does not satisfy
\[
\lim_{t\to +\infty}w(x,t)=u_{1}(x)\,\,\, \text{in}\,\,\, \Omega.
\]
\end{itemize}
\end{theorem}

\begin{proof}
The proof of items (i) through (iv) follows similarly to the $p$-Laplacian case in \cite{HajWang19}, with only minor adjustments. We therefore highlight the modifications required for item (iv), as the remaining cases are entirely analogous.

Since we also have the comparison principle in this setting, we may assume that $v_0$ is smooth. Then, for a constant $\mathrm{M} > 0$, depending on universal parameters and on $v_0$, we have for any nonnegative test function $\varphi$, independent of $t$ and a a sub-domain $V\subset\subset \Omega$ such that $\varphi$ is supported in $V$, that
\[
\int_V \frac{g(|\nabla v_0|)}{|\nabla v_0|} \nabla v_0 \cdot \nabla \varphi \geq \mathrm{M} \int_V \varphi \, .
\]
The continuity of the gradient $\nabla w$ up to $t = 0$, as stated in \cite{Lie93} allows us to obtain
\[
\int_V \frac{g(|\nabla w|)}{|\nabla w|} \nabla w \cdot \nabla \varphi \geq \frac{\mathrm{M}}{2} \int_V \varphi
\]
for any $t \in (0, t_0)$ and any nonnegative function $\varphi$, independent of $t$, supported in $V$, and satisfying the condition
\begin{equation}\label{eq7}
    \frac{\int_V |\nabla \varphi|}{\int_V \varphi} \leq \Lambda,
\end{equation}
for some fixed constant $\Lambda > 0$ and some $t_0 > 0$ depending on $\Lambda$. Then, the subsolution condition for $w$ yields that for any $t_2 > t_1$ in $(0, t_0)$,
\[
\int_V w \varphi \Big|_{t_1}^{t_2} \leq 0,
\]
for any nonnegative function $\varphi$, supported in $V$, independent of $t$, and satisfying condition \eqref{eq7}. Consequently, we obtain
\[
w(x,t_2) \leq w(x,t_1)
\]
for all $x \in \Omega$ and $t_2 \geq t_1 > 0$. Then the Parabolic Comparison Principle \ref{comppri} readily implies $w$ is decreasing in $t$ for $t\in (0,\infty).$
The remainder of the proof follows the same way as in the proof of Theorem 4.1 in \cite{HajWang19}.

Now, we prove the item (v). Given  a $\eta>0$ consider $v_{0}$ such that $\|v_{0}-u_{1}\|_{L^{\infty}(\Omega)}<\eta$ and $\mathcal{J}_{\varepsilon}[v_{0}]<\mathcal{J}_{\varepsilon}[u_{1}].$ Instead of considering the classical regularization of the $p$-Laplacian operator as in \cite{HajWang19}, when dealing with the prototype $g$-Laplacian associated with an $N$-function $G$ satisfying the Lieberman conditions, it is more convenient to use the regularization introduced by Maz'ya and Cianchi in \cite{Cianchi}.
More precisely, we consider a family of standard mollifiers $(\eta_\tau)_{\tau>0}\subset C^{\infty}(\mathbb{R})$ such that 
\[
\int_{\mathbb{R}} \eta_\tau(t)\,dt = 1,
\quad \text{and} \quad
\operatorname{supp}(\eta_\tau) \subset (-\tau,\tau).
\]
Now, as in the Section \ref{Section3}, let \( F(t) := \frac{g(t)}{t} \) and define \( \Phi : \mathbb{R} \to [0,+\infty) \) by
\[
\Phi(s) := F(e^s).
\]
We then define \( \Phi_\tau := \eta_\tau * \Phi \). Lastly, we set
\[
\begin{aligned}
f_\tau:(0,\infty)&\to\mathbb{R},\\
t&\mapsto f_\tau(t)=\Phi_\tau(\log t).
\end{aligned}
\]
Following this procedure, we are able to define the family of N-functions \( \{G_\tau\}_{\tau>0} \) by setting
\[
G_\tau(t) := \int_0^t g_\tau(s)\,ds,
\]
where \( g_\tau(t) := F_\tau(t)\cdot t \) and
\[
F_\tau(t) :=
\frac{f_\tau(\sqrt{\tau + t^2}) + \tau}
{1 + \tau \cdot f_\tau(\sqrt{\tau + t^2})}.
\]
By \cite{Cianchi}, the sequence of functions $(F_{\tau})_{\tau>0}$ satisfies:
\begin{itemize}\label{itens}
\item[(A)] \( F_\tau \in C^\infty(0,+\infty) \);
\item[(B)] \( \tau \leq F_\tau(t) \leq \frac{1}{\tau} \) for all \( t > 0 \);
\item[(C)] \( \min\{\delta_0 - 1,0\} \leq \frac{t \cdot F_\tau'(t)}{F_\tau(t)} \leq \max\{g_0 - 1,0\} \);
\item[(D)] \( \lim\limits_{\tau \to 0} F_\tau(|\xi|)\xi = F(|\xi|)\xi \), uniformly in \( \mathrm{B}_{r} \) for any \( R > 0 \).
\end{itemize}
In particular, \( G_\tau \in C^2[0,\infty) \cap C^\infty(0,+\infty) \), such that \( G_\tau \to G \) and \( g_\tau \to g \) uniformly on compact subsets of \( [0,\infty) \). Moreover, \( g_\tau \) satisfy Lieberman's conditions, with
\begin{equation}\label{conddelieberman}
\min\{\delta_0,1\} \leq \frac{t \cdot g_\tau'(t)}{g_\tau(t)} \leq \max\{g_0,1\}.
\end{equation}
 We conclude that \( G_\tau \) satisfies the same structural conditions as \( G \). Also, \( g_\tau(1) \sim g(1) \) for \( \tau \in (0,1) \) small enough.

 Consider the problem \eqref{evprobprincipal} with $v_{0}$ as the initial data, we may change the value of $v_0$ slightly if necessary so so that there exists $\tau_{0}>0$ sufficiently small such that $v_{0}$ is not a solution of the following regularized problem.
\begin{equation*}
-\Delta_{g_\tau}u+a(x)\beta_{\varepsilon}(u)=0\quad\text{in}\quad\Omega, 
\end{equation*}
for any $\tau \in (0,\tau_{0})$. Let $w^{\tau}$ be the smooth solution of the uniformly parabolic boundary-value problem
\begin{equation*}\label{regprobl}
\begin{cases}
    w_t - \Delta_{g_\tau} w + a(x)\,\beta_\varepsilon(w) = 0 
    & \text{in } \Omega \times (0,\infty), \\[4pt]
    w(x,t) = \psi(x) 
    & \text{on } \partial\Omega \times (0,\infty), \\[4pt]
    w(x,0) = v_0(x) 
    & \text{in } \overline{\Omega}.
\end{cases}\tag{$\mathrm{P}_{\tau}$}
\end{equation*}
We define the functional
 $$\mathcal{J}_{\varepsilon}^{\tau}[u]\defeq \int_\Omega G_\tau(|\nabla u|) + a(x)\Gamma_{\varepsilon}(u)\,dx.$$
 On the other hand, since $w^{\tau}$ is a solution of the problem \eqref{regprobl},  we have
 \begin{equation}\label{5.5}
 \int_0^T\int_\Omega ((w^{\tau})_{t})^2 - \operatorname{div}\left(\frac{g_\tau(|\nabla w^{\tau}|)}{|\nabla w^{\tau}|}\nabla w^{\tau} \right)(w^{\tau})_{t} + a(x)\beta_{\varepsilon}(w^{\tau})(w^{\tau})_t = 0.
 \end{equation}
Using that $(w^{\tau})_t=0$ on $\partial\Omega\times (0,\infty)$, we obtain from \eqref{5.5}, by integration by parts, that
\[
\int_0^T\int_\Omega ((w^{\tau})_{t})^2 + \frac{g_\tau(|\nabla w^{\tau}|)}{|\nabla w^{\tau}|}\nabla w^{\tau}\cdot\nabla w^{\tau} + a(x)\Gamma_{\varepsilon}(w^{\tau})_t = 0,
\]
which implies that
\[
\int_0^T\int_\Omega ((w^{\tau})_t)^2 + \left(G_\tau(|\nabla w^{\tau}|)\right)_t + a(x)\Gamma_{\varepsilon}(w^{\tau})_t = 0.
\]
By integrating over $(0,T)$, we obtain
\begin{align*}
\int_0^T\int_\Omega ((w^{\tau})_t)^2 + \int_\Omega G_\tau(|\nabla w^{\tau}(x,T)|) + \int_\Omega a(x)\Gamma_{\varepsilon}(w^{\tau}(x,T))=\int_\Omega G_\tau(|\nabla w^{\tau}(x,0)|) + \int_\Omega a(x)\Gamma_{\varepsilon}(w^{\tau}(x,0)).
\end{align*}
That is,
$$\int_0^T\int_\Omega((w^{\tau})_t)^2 +\mathcal{J}^{\tau}_{\varepsilon}[w^{\tau}(\cdot,T)]= \mathcal{J}^{\tau}_{\varepsilon}[w^{\tau}(\cdot,0)]=\mathcal{J}^{\tau}_{\varepsilon}[v_{0}],$$
since $w^{\tau}$ is solution to \eqref{regprobl}. This implies that
$$\mathcal{J}^{\tau}_{\varepsilon}[w^{\tau}(\cdot,T)]\leq \mathcal{J}^{\tau}_{\varepsilon}[v_0].
$$
We claim that there exists a sequence $\tau_k \to 0^+$ such that $w^{\tau_k}$ converges to $w$ solution of the problem \eqref{evprobprincipal} in $C^1$ norm for each $\Omega_T$ and $T>0$. The proof strategy is similar to that in \cite{Cianchi} for the elliptic case. First, note that since $u_2 \leq u_1 \leq u_0$, we may choose initial data satisfying $u_2 \leq v_0 \leq u_0$. Then, by the Parabolic Comparison Principle \ref{comppri}, we obtain the bound $u_2\leq w \leq u_0$ in $\Omega\times(0,+\infty)$. Therefore, for all $\tau \in (0,1)$
\[
\|w^{\tau}\|_{L^{\infty}(\Omega_T)} \leq \|u_2\|_{L^{\infty}(\Omega)}+ \|u_0\|_{L^{\infty}(\Omega)}.
\]
Consequently, by the $C^{1,\alpha}$ estimate from \cite{Lie93}, for each $T>0$ there exists a universal constant $\mathrm{C}>0$, independent of $\tau$ (see \eqref{conddelieberman}) such that
\begin{equation}\label{eq 3}
\|\nabla w^{\tau}\|_{C^\alpha(\Omega_T)}\leq \mathrm{C}.
\end{equation}
Therefore, there exists a sequence $\tau_k \to 0^+$ such that $w^{\tau_k} \to \bar{w}$ in $C^1(\Omega_T)$.

On the other hand, by (D) and \eqref{eq 3}, it follows that the sequence
\[
\frac{g_\tau(|\nabla w^{\tau}|)}{|\nabla w^{\tau}|}\nabla w^{\tau}
\]
is bounded in $L^2(0,T;L^G(\Omega))$. Hence, there exist a (not relabeled) subsequence $\tau_k$ and a vector-valued function $\bar{u} \in L^2(0,T;L^G(\Omega))$ such that
\begin{equation}\label{eq 4}
\frac{g_{\tau_k}(|\nabla w^{\tau_k}|)}{|\nabla w^{\tau_k}|}\nabla w^{\tau_k} \rightharpoonup \bar{u}\quad\text{in}\,\,\,L^2(0,T;L^G(\Omega)).
\end{equation}
Since $w^{\tau_k} \to \bar{w}$ and $\nabla w^{\tau_k} \to \nabla \bar{w}$ converges point-wise, we deduce that
\begin{equation}\label{eq 5}
\frac{g(|\nabla \bar{w}|)}{|\nabla \bar{w}|}\nabla \bar{w} = \bar{u}.
\end{equation}
Now, for any test function $\varphi$, we have
$$\int_0^t\int_\Omega\left(-w^{\tau_k}\varphi_t + \frac{g_{\tau_k}(|\nabla w^{\tau_k}|)}{|\nabla w^{\tau_k}|}\nabla w^{\tau_k}\cdot\nabla\varphi + a(x)\beta_{\varepsilon}(w^{\tau_k})\varphi \right)=0. $$
Due to \eqref{eq 4} and \eqref{eq 5}, by passing to the limit as $\tau_k \to 0^+$, we can deduce that
$$
\int_0^T\int_\Omega\left(-\bar{w}\varphi_t + \frac{g(|\nabla \bar{w}|)}{|\nabla \bar{w}|}\nabla \bar{w}\cdot\nabla\varphi + a(x)\beta_{\varepsilon}(\bar{w})\varphi \right)=0,
$$
and the boundary condition implies that $\bar{w}$ is a solution of
 \[
\begin{cases}
    w_t - \Delta_{g_\tau} w + a(x)\,\beta_\varepsilon(w) = 0 
    & \text{in } \Omega_T:=\Omega \times (0,T), \\[4pt]
    w(x,t) = \psi(x) 
    & \text{on } \partial\Omega \times (0,T), \\[4pt]
    w(x,0) = v_0(x) 
    & \text{in } \overline{\Omega}.
\end{cases}
\]
By uniqueness (Corollary \ref{Uniquenesscor}), it follows that $\bar{w} = w$ in $\Omega \times (0,T)$ for every $T>0$, which completes the proof of the claim.

Therefore $\mathcal{J}_{\varepsilon}[w(\cdot, t)]\leq \mathcal{J}_{\varepsilon}[v_0]< \mathcal{J}_{\varepsilon}[u_1].$
In conclusion, $w(\cdot, t)$ does not converge to $u_1$ as $t\to \infty.$
\end{proof}


\subsection*{Acknowledgments}

\hspace{0.4cm} J.S. Bessa has been supported by FAPESP-Brazil under Grant No. 2023/18447-3.

\end{document}